\documentclass[11pt]{amsart}
\usepackage{amssymb,amsfonts,dsfont}
\usepackage{graphicx}
\usepackage{amsmath}
\usepackage{latexsym}
\usepackage{color}
\usepackage{comment}

\newtheorem{theorem}{Theorem}[section]
\theoremstyle{plain}
\newtheorem{lemma}[theorem]{Lemma}
\newtheorem{proposition}[theorem]{Proposition}

\newtheorem{corollary}[theorem]{Corollary}

\theoremstyle{remark}
\newtheorem{definition}[theorem]{Definition}

\numberwithin{equation}{section}

\DeclareMathOperator{\supp}{supp}

\newcommand{\M}{\mathcal{M}}

\newcommand{\Complex}{\mathbb{C}}

\newcommand{\abs}[1]{\left\vert#1\right\vert}

\newcommand{\norme}[1]{\|#1\|_E}
\newcommand{\normcomm}[1]{\|#1\|_{E(\mathcal{M},\tau)}}

\newcommand{\nonsp}{E(\mathcal{M},\tau)}

\def\underset#1\to#2{\mathop{#2}\limits_{#1}{ }}
\def\overset#1\to#2{\mathop{#2}\limits^{#1}{ }}

\newcommand{\one}{\textup{\textbf{1}}}
\newcommand{\tauone}{\tau(\textup{\textbf{1}})}

\newcommand{\Mtau}{\left(\mathcal{M},\tau\right)}

\begin{document}
\title[Banach envelopes in symmetric spaces of measurable operators]
{Banach envelopes in symmetric spaces of measurable operators}

\author{M.M. Czerwi\'nska}
\address{ Department of Mathematics and Statistics, University of North Florida, Jacksonville, FL 32224} \email{m.czerwinska@unf.edu}

\author
{A. Kami\'nska}
\address{Department of Mathematical Sciences, The University of
Memphis, Memphis, TN 38152} \email{kaminska@memphis.edu}

\thanks {\emph{2010 subject classification}\ {46B20, 46B28, 46B40, 46E30, 46L52, 47L20}}


\maketitle
 \begin{abstract}
 We study Banach envelopes for commutative symmetric sequence or function spaces, and noncommutative symmetric spaces of measurable operators. 
 We characterize the class $(HC)$ of quasi-normed symmetric sequence or function spaces $E$ for which their Banach envelopes $\widehat{E}$ are also symmetric spaces. The class of symmetric spaces satisfying $(HC)$ contains but is not limited to order continuous spaces. 
      Let $\M$ be a non-atomic, semifinite von Neumann algebra with a faithful, normal, $\sigma$-finite trace $\tau$ and $E$ be as symmetric function space on $[0,\tauone)$ or symmetric sequence space. We compute Banach envelope norms on $\nonsp$ and $C_E$ for any quasi-normed symmetric space $E$. Then we show under assumption that $E\in (HC)$ that the Banach envelope $\widehat{\nonsp}$ of $\nonsp$ is equal to $\widehat{E}\Mtau$ isometrically. We also prove the analogous result for unitary matrix spaces $C_E$. \end{abstract}
\section{Preliminaries}

The concept of the Banach envelope was independently introduced by Peetre \cite{P} in 1974, and Shapiro \cite{Sh} in 1977.
Since then this concept have been studied by many authors in different spaces and settings. 
  In 1977, Kalton \cite{K1}  gave an isomorphic and  isometric representation of the Banach envelop of non-locally convex Orlicz spaces.  This result   was later generalized by Drewnowski in \cite{D} to Musielak-Orlicz spaces. In \cite{K2}, Kalton obtained the similar result in Orlicz sequence spaces $\ell_F$ and applied it to studies of isomorphic copies of $l_p$, $0<p<\infty$, in non-locally convex Orlicz spaces. This also inspired further  studies  in \cite{DN}, where a characterization of the Banach envelop of non-locally convex Orlicz spaces was obtained without the assumption of separability.

In the eighties Cwikel and  Fefferman in \cite{CF1980, CF1984} computed the Banach envelop semi-norm in the weak $L^1$ space called also Marcinkiewicz  space, and showed that the space equipped with  this semi-norm is not complete. In 2008, Kalton and Sukochev \cite{KS}  calculated the Banach envelop norm in the sequence Marcinkiewicz space $\ell_{1,\infty}$, and very recently more studies of this subject was continued by Pietsch in \cite{P}.

  It is also worth to notice that in \cite{Kal78}, Kalton proved that 
  if $X$ is a quasi-Banach space with separating dual then its Banach envelope has no type $p$ for any $p>1$.    It follows that for instance  the spaces $l_p$,  $L_p$ for $1<p<\infty$, or the Orlicz  spaces $\ell_\varphi$, $L_\varphi$ with the lower Orlicz-Matuszewska index $\alpha_\varphi >1$, cannot be Banach envelops of $X$.  In the papers \cite{kammast, kamlin} there have been given general descriptions of Banach envelopes for symmetric function spaces whose cones of decreasing or non-negative elements are $1$-concave. The latter two results have been applied for finding characterization of Banach envelops in Orlicz-Lorentz and weighted Ces\'aro spaces, respectively. 
  Finally in \cite{kamray}  there is a description of the Banach envelope of a Marcinkiewicz class in terms of Hardy-Littlewood submajorization.
  
  In this note we will study the Banach envelop of quasi-normed non-commutative spaces of measurable operators and the unitary matrix spaces \cite{A, DDP2, DDP4, noncomm}.  Our results, among others, will generalize and extend the characterization given by Nawrocki in \cite{naw},  where he investigated the  separable unitary matrix spaces.

Recall that a \textit{quasi-normed space} is a vector space $(X,\|\cdot\|_X)$ over a field of complex numbers $\mathbb{C}$ equipped with a function $\|\cdot\|_X:X\to \mathbb R$ satisfying
\begin{itemize}
\item[{(i)}] $\|x\|_X>0$ for all $x\neq 0$,
\item[{(ii)}] $\|\alpha x\|_X=\abs{\alpha} \|x\|_X$ for all $\alpha\in \Complex$ and all $x\in X$, and
\item[{(iii)}] there exists $C>0$ such that  $\|x+y\|_X\leq C\left(\|x\|_X+\|y\|_X\right)$ for all $x,y\in X$.
\end{itemize}

Denoting by $ \mbox{co} (B_X)$ the convex hull  of the unit ball $B_X$ in $X$, let $\|x\|_{\widehat{X}}=\inf\{\lambda>0: x/\lambda\in \mbox{co} (B_X)\}$ be the Minkowski functional of $B_X$.
Clearly $\|\cdot\|_{\widehat{X}}$ is a semi-norm, and if the dual $X^*$ separates the points of $X$, that is for every $0\ne x\in X$ there is $x^*\in X^*$ such that $x^*(x)\neq 0$, it is in fact a norm. Setting  $N =\{x\in X : \|x\|_{\widehat{X}}=0\}$, $X/N$  is a normed linear space, whose completion is called the \textit{Banach envelope} $\widehat{X}$ or $X\widehat{\phantom{x}}$ of $X$ \cite{KPR}. 

Obviously $\|x\|_{\widehat{X}}\leq \|x\|_X$ for any $x\in X$, and so the inclusion map $i:(X,\|\cdot\|_X)\to (X,\|\cdot\|_{\widehat{X}})$ is continuous with its range dense in $\widehat{X}$. 
Moreover, in view of the relation that $z\in \mbox{co}(B_X)$ if and only if $\|z\|_{\widehat{X}}\leq 1$, we have
\begin{equation}\label{eq:01}
\|x^*\|=\sup_{\|z\|_X\leq 1}|x^*(z)|=\sup_{z\in \mbox{co}(B_X)}|x^*(z)|=\sup_{\|z\|_{\widehat{X}}\leq 1}|x^*(z)|.
\end{equation}
Thus any $x^*\in X^*$ is bounded on $(X,\|\cdot\|_{\widehat{X}})$ as well and by the density it can be extended to $\widehat{X}$ with preservation of the norm. On the other hand, if $x^*\in \widehat{X}^*$ then its restriction to $X$ is bounded on $(X,\|\cdot\|_X)$.

Another way of introducing $\widehat{X}$ is through the canonical embedding $J$ of $X$ into its second dual $X^{**}$, defined as usual by $Jx(x^*)=x^*(x)$. If we assume that $X^*$ separates points in $X$ then $J$ is injective.
Moreover,
$\|J\|=\sup_{\|x\|_X\leq 1}\|Jx\|_{X^{**}}=\sup_{\|x\|_X\leq 1}\sup_{\|x^*\|\leq 1}|x^*(x)|\leq 1.$
Since the Hahn-Banach theorem is usually not satisfied in the spaces that are not locally convex, we do not have that $\|J\|=1$, so $J$ is not isometry.  The Banach envelope of $X$ can be taken to be the closure in $X^{**}$ of the image of $X$ under $J$. In fact since $\|\cdot\|_{\widehat{X}}$ is a norm, by the Hahn-Banach Theorem,
\[
\|x\|_{\widehat{X}}=\sup_{\|x^*\|\leq 1}|x^*(x)| =\|Jx\|_{X^{**}}.
\]
Therefore $\widehat{X}=\overline{J(X)}^{X^{**}}$  \cite{KPR}. 

The following very useful formula of the Banach envelope norm was introduced by Peetre in \cite{P}. For $x\in X$, set 
\[
\|x\|^*=\inf\Bigl\{\sum_{i=1}^n \|x_i\|_{X}:\sum_{i=1}^n x_i=x,\  x_i \in X, \ n\in\mathbb{N}\Bigr\}.
\]
It is standard to check that $\|\cdot\|^*$ is a seminorm on $X$  and
\begin{equation}
\label{eq:strongest}
\|x\|^*\leq \|x\|_X\quad \text{for every}\quad  x\in X.
\end{equation}
In fact $\|\cdot\|^*$ is the strongest seminorm for which (\ref{eq:strongest}) holds. That is, if $\|\cdot\|^0$ is another seminorm on $X$ satisfying $\|x\|^0\leq \|x\|_X$ for all $x\in X$, then $\|x\|^0\leq \|x\|^*$ for all $x\in X$. Indeed, if $x=\sum_{i=1}^n x_i$, $x_i\in X$, then $\|x\|^0=\|\sum_{i=1}^n x_i\|^0\leq \sum_{i=1}^n \|x_i\|^0\leq \sum_{i=1}^n \|x_i\|_X$, and so $\|x\|^0\leq \|x\|^*$.  Thus in particular, 
\begin{equation}
\label{eq:02}
\|x\|_{\widehat{X}}\leq \|x\|^*, \ \ \ x\in X.
\end{equation}
Applying (\ref{eq:strongest}) and the defintion of $\|\cdot\|^*$ one can  easily show that   the sets of  bounded functionals on $(X,\|\cdot\|^*)$ and on $(X,\|\cdot\|_X)$ coincide. Therefore
for any $x^*\in (X,\|\cdot\|_X)^*=(X,\|\cdot\|_{\widehat{X}})^*=(X,\|\cdot\|^*)^*$, in view of (\ref{eq:01}), (\ref{eq:strongest}) and (\ref{eq:02}) we get 
\[
\|x^*\| = \sup_{\|x\|_X\leq 1}|x^*(x)|=\sup_{\|x\|^*\leq 1}|x^*(x)|= \sup_{\|x\|_{\widehat{X}}\leq 1}|x^*(x)|.
\]
 Hence by the Hahn-Banach theorem, 
 \[
 \|x\|_{\widehat{X}}=\sup_{\|x^*\|\leq 1}|x^*(x)|=  \|x\|^*.
 \]

For $0<p<1$, the Banach envelope of the function space $L_p[0,1]$ is trivial $\{0\}$, while of the sequence space $\ell_p$ is equal to $\ell_1$ \cite{KPR}.

In this note we investigate Banach envelopes of quasi-normed  symmetric function and sequence spaces $E$, of unitary matrix spaces $C_E$, and of noncommutative spaces $\nonsp$ of $\tau$-measurable operators corresponding to $E$ and associated to a non-atomic von Neumann algebra $\M$ equipped with a normal, faithful, $\sigma$-finite trace $\tau$. In the first part of the manuscript we focus on the commutative case. We define the property $(HC)$ on $E$, which is necessary and sufficient for the Banach envelope $\widehat{E}$ to maintain the sequence or function character of $E$. We further show that if $E$ is a symmetric spaces with property $(HC)$ then its Banach envelope $\widehat{E}$ is also a symmetric space. The class of symmetric spaces satisfying $(HC)$ is quite large, as it contains but is not restricted to all order continuous spaces.
 In the second part of the paper we characterize Banach envelopes of noncommutative spaces $\nonsp$ and  $C_E$ in terms of the Banach envelope of $E$.  We compute  the Banach envelope norms on  $\nonsp$ and  $C_E$,  without any extra assumptions on the quasi-normed symmetric function or sequence spaces $E$.  For $E\in (HC)$, we obtain the complete characterizations of Banach envelopes of $\nonsp$ or $C_E$, that is we show that $\left(\widehat{\nonsp},\|\cdot\|_{\widehat{\nonsp}}\right)=\left(\widehat{E}\Mtau,\|\cdot\|_{\widehat{E}\Mtau}\right)$ and $\left(\widehat{C_E},\|\cdot\|_{\widehat{C_E}}\right)=\left(C_{\widehat{E}},\|\cdot\|_{C_{\widehat{E}}}\right)$.   Analogous result has been shown for the unitary matrix spaces $C_E$ in \cite{naw}, where $E$ was assumed to be weakly-separable quasi-normed symmetric sequence space. Our claim for noncommutative symmetric sequence spaces will yield a stronger version of the result in \cite{naw}.

Throughout the paper $\M$ will stand for a semi-finite von Neumann algebra acting on the Hibert space $H$, with a normal, faithful, semi-finite trace $\tau$. The identity on $\M$ will be denoted by $\one$ and the lattice of projections in $\M$ by $P(\M)$.  Non-zero projection $p\in
P(\mathcal{M})$ is called \emph{minimal} if $q\in P(\mathcal{M})$ and $q\leq p$ imply that $q=p$ or
$q=0$. The von Neumann algebra $\M$ is called \emph{non-atomic} if it has no minimal orthogonal projections. We say that the trace $\tau$ on $\mathcal{M}$ is $\sigma$-\textit{finite} if there exists a sequence $\{p_n\}\subseteq P(\mathcal{M})$ such that $p_n\uparrow \one$ and $\tau(p_n)<\infty$.   A closed and densely defined operator $x$ with domain $D(x)$ is called \textit{affiliated} with $\M$ whenever $ux=xu$ for all unitary operators $u\in\M'$. An operator $x$ is \textit{called $\tau$-measurable}, if $x$ is affiliated with $\M$ and its spectral projection $e^{\abs{x}}(\cdot)$ satisfies $\tau(e^{\abs{x}}(\lambda,\infty))<\infty$, for some $\lambda>0$. The set $S\Mtau$ of all $\tau$-measurable operators forms a $*$-algebra  under the strong sum and strong multiplication operations. 
For a closed and densely defined operator $x$, the \textit{support projection} is given by $s(x)=e^{\abs{x}}(0,\infty)$ and the \textit{null projection} by $n(x)=\one-s(x)=e^{\abs{x}}\{0\}$.

For $x\in S\Mtau$, the function
\[
\mu(t,x)=\inf\{s\geq 0:\, \tau(e^{\abs{x}}(s,\infty))\leq t\},\quad t\geq 0,
\]
is called a \textit{generalized singular value function} or \textit{decreasing rearrangement} of $x$. The distribution function of $x$ at $s\geq 0$ is given by $d(s, x)= \tau(e^{\abs{x}}(s,\infty))$.
We  set $\mu(\infty,x)=\lim_{t\to\infty}\mu(t,x)$, $x\in S\Mtau$. Observe that $\mu(t,x)=0$ for all $t\geq \tauone$. Hence we  will consider $\mu(\cdot,x)$ as a function on $[0,\tauone)$.

 A measurable operator $x$ is called \textit{positive} if $\langle x\xi,\xi\rangle\geq 0$ for all $\xi\in D(x)$. The partial order in the algebra of self-adjoint operators generated by the cone of positive operators will be denoted by $\leq$. 

Let $I$ denote the interval $[0,\tauone)$, $\tauone\leq \infty$, with the Lebesgue measure $m$, or the set $\mathbb N$ with the counting measure $m$. Let $L_0(I)$ be a collection of all complex-valued  measurable functions (resp. sequences) whenever  $I=[0,\tauone)$ (resp. $I=\mathbb N$).
For $f\in L_0(I)$, the \textit{support} of $f$, that is the set of all $t \in I$ for which $f (t)\neq 0$, will be denoted by $\supp{f}$.

Denote by $\left(L_1\Mtau,\|\cdot\|_{L_1\Mtau}\right)$ the Banach space of all $\tau$-integrable operators, that is the space of all $x\in S\Mtau$ for which $\|x\|_{L_1\Mtau}=\tau(\abs{x})<\infty$. The space $L_1\Mtau +\M$ is the Banach space of all measurable operators $x$ for which $\mu(x)\in L_1+L_{\infty}$, where $L_1=L_1[0,\tauone)$ and $L_\infty=L_\infty[0,\tauone)\subseteq L_0=L_0[0,\tauone)$ are the spaces of integrable and bounded functions on $[0,\tauone)$, respectively.

 In the commutative case, that is when $\M$ is identified with $ L_\infty$ with the trace $\tau$ given by integration, the algebra $S\Mtau$ coincides with the algebra of all functions $f\in L_0$ which  are finite except of the sets of finite measure. 
The singular value function $\mu(f)$ coincides then with the decreasing rearrangement of a function $f$ commonly denoted in the literature by $f^*$ \cite{BS}. Moreover, the distribution function for the commutative case is given by $d(s, f)=m\{t\geq 0:\, |f|(t)>s\}$, $s\geq 0$.

In the case of discrete measure, for $f=\{f(n)\}_{n=1}^\infty\in L_0(\mathbb N)$  with $\lim_{n\to\infty} f(n) =0$,  $\mu(t,f)$ is a  finite and countably valued function on $[0,\infty)$. In this case we will identify its decreasing rearrangement $\mu(f)$ with the sequence $\left\{\mu(n-1, f)\right\}_{n=1}^\infty$.

A quasi-normed space $\left(E,\|\cdot\|_E\right)$ is called a \textit{symmetric space} if $E\subseteq L_0(I)$ and for any $f\in L_0(I)$, $g\in E$, by $\mu(f)\leq \mu(g)$ it follows that $f\in E$ and $\|f\|_E\leq \|g\|_E$.
 When $I=[0,\tauone)$, $\tauone\leq\infty$, (resp. $I=\mathbb N$) a quasi-normed space $E$ will be called a quasi-normed \textit{function} (resp. \textit{sequence}) space.  If $(E,\|\cdot\|_E)$ is a normed symmetric space it is well known that $(L_1\cap L_{\infty} )(I)\hookrightarrow E\hookrightarrow (L_1+L_\infty)(I)$ \cite{KPS}. In particular when $I=\mathbb N$, $\ell_1\hookrightarrow E\hookrightarrow \ell_\infty$.  Notice also  that in this case  the condition $E\subseteq c_0$ is equivalent to $E \ne\ell_\infty$.

Recall that  $f\in E$ is called  order continuous whenever for any $ f_n \downarrow 0$ a.e. with $0\leq f_n \leq f$ we have $\|f_n\|_E\downarrow 0$. If all elements of $E$ are order continuous then the space $E$ is said to be \textit{order continuous}.
 
Given a symmetric quasi-normed function space $E\subseteq L_0[0,\tauone)$, the \textit{symmetric space of measurable operators} is defined as
\[
\nonsp = \{ x\in S\Mtau: \ \mu(x)\in E\}.
\]
 Similarly,  if $E$ is a quasi-normed symmetric sequence space such that   $E\subseteq c_0$, then  the {\em  unitary matrix space} is defined by 
\[
C_E=\{x\in K(H):\, S(x)\in E\},
\]
 where $K(H)$ is the ideal of all compact operators on a Hilbert space $H$ and $S(x)=\{S_n(x)\}_{n=1}^\infty$ is a sequence of singular numbers of an operator $x$.
 
 If $(E, \|\cdot\|_E)$ is a Banach symmetric space then $\|x\|_{\nonsp} = \|\mu(x)\|_E$ or  $\|x\|_{C_E} = \|S(x)\|_E$ is a norm and the spaces $\nonsp$ or $C_E$ equipped with these norms are complete \cite{KS}.

Moreover, the following holds.  
\begin{lemma}\cite[Lemma 6]{S}
\label{lem:quasi}
Let $E$ be a symmetric quasi-normed function space or a symmetric quasi-normed sequence space such that $E\subseteq c_0$. Then the functional  $\normcomm{x}=\norme{\mu(x)}$, $x\in\nonsp$, or $\|x\|_{C_E}=\norme{S(x)}$, $x\in C_E$, is a quasi-norm.
\end{lemma}

For more information on the noncommutative symmetric spaces we refer the reader to \cite{DDP2, noncomm,LSZ}, on operator algebras to \cite{T}, and on Banach latices to \cite{BS, KPS}. 

Before we proceed to the main results of the paper, we state a few facts for easy reference.

\begin{lemma}\cite[Corollary 1.6]{CK-kext}
\label{lm:01} Let $x\in S\Mtau$ and $p\in P(\M)$. If $px=xp=0$ and $0\leq C\leq \mu(\infty,x)$ then $\mu(x+Cp)=\mu(x)$.
\end{lemma}

\begin{proposition}\cite[Corollary 1.10]{CK-kext}
\label{prop:isom}
Let $\mathcal{M}$ be a non-atomic von Neumann algebra with a $\sigma$-finite trace $\tau$, $x\in S\Mtau$, and $\abs{x}\geq \mu(\infty,x)s(x)$. Denote by $p=s(\abs{x}-\mu(\infty,x)s(x))$ and define projection $q\in P(\M)$ in the following way.
\begin{itemize}
\item[{(i)}] If $\tau(s(x))<\infty$ set $q=\one$.
\item[{(ii)}] If $\tau(s(x))=\infty$ and $\tau(p)<\infty$, set $q=s(x)$.
\item[{(iii)}] If $\tau(p)=\infty$, set $q=p$.
\end{itemize}
 Then there exist a non-atomic commutative von Neumann subalgebra $\mathcal{N}\subseteq q\M q$ and a unital $*$-isomorphism
$V$ acting from the $*$-algebra $S\left(\left[0,\tauone\right),m\right)$ into the $*$-algebra $S(\mathcal{N},\tau)$, such that 
\[ V\mu(x)=\abs{x}q\ \ \ \text{ and }\ \ \ \mu(Vf)=\mu(f)\ \ \text{ for all } f\in S\left([0,\tauone),m\right).
\]
\end{proposition}

\begin{corollary}
\label{cor:isom}
Let $\mathcal{M}$ be a non-atomic von Neumann algebra with a $\sigma$-finite trace $\tau$ and $x\in S\Mtau$ and $r=e^{\abs{x}}(\mu(\infty,x),\infty)$.  
Set $q=\one$ whenever $\tau(r)<\infty$, and $q=r$ if $\tau(r)=\infty$. 

Then there exist a non-atomic commutative von Neumann subalgebra $\mathcal{N}\subseteq q\M q$ and a unital $*$-isomorphism
$V$ acting from the $*$-algebra $S\left(\left[0,\tauone\right),m\right)$ into the $*$-algebra $S(\mathcal{N},\tau)$, such that 
\[ V\mu(x)=\abs{x}r+\mu(\infty, x)V\chi_{[\tau(r),\infty)}\ \ \ \text{ and }\ \ \ \mu(Vf)=\mu(f)\ \ \text{ for all } f\in S\left([0,\tauone),m\right).
\]
\end{corollary}
\begin{proof}
Consider the operator $x_0=\abs{x}r$, where we have $s(x_0)=r$ and $x_0\geq \mu(\infty, x) s(x_0)$. Moreover, $\mu(x_0)=\mu(x)\chi_{[0,\tau(r))}$ \cite{noncomm}. If $\tau(r)<\infty$, then $\mu(\infty, x_0)=0$. Otherwise $\mu(x_0)=\mu(x)$. In either case $x_0\geq \mu(\infty, x_0)s(x_0)$.  Moreover, $p=s(x_0-\mu(\infty, x_0) s(x_0))=e^{x_0}(\mu(\infty, x_0),\infty)=e^{x}(\mu(\infty, x),\infty)=r$. If $\tau(r)=\infty$ set $q=p=r$, and if $\tau(r)<\infty$, $q=\one$. By Proposition \ref{prop:isom} (i) and (iii) applied to $x_0$ there exist a non-atomic commutative von Neumann subalgebra $\mathcal{N}\subseteq q\M q$ and a unital $*$-isomorphism
$V$ acting from the $*$-algebra $S\left(\left[0,\tauone\right),m\right)$ into the $*$-algebra $S(\mathcal{N},\tau)$, such that 
\[ V\mu(x_0)=x_0q\ \ \ \text{ and }\ \ \ \mu(Vf)=\mu(f)\ \ \text{ for all } f\in S\left([0,\tauone),m\right).
\]

In case of $\tau(r)=\infty$, $\mu(x)=\mu(x_0)$ and $q=r$, and therefore  $V\mu(x)=x_0r=\abs{x}r$.

Consider now the case when $\tau(r)=\tau(e^{\abs{x}}(\mu(\infty,x),\infty))<\infty$ with $q=\one$. Since $\mu(\infty, x)=\inf\{s\geq 0:\, \tau(e^{\abs{x}}(s,\infty))<\infty\}$, we have that $\tau(e^{\abs{x}}(s,\infty))=\infty$ for all $s\in[0, \mu(\infty,x))$. Recalling the definition of $\mu(t,x)=\inf\{s\geq 0:\, \tau(e^{\abs{x}}(s,\infty))\leq t\}$, it is easy to observe that $\mu(t,x)=\mu(\infty, x)$ for all $t\geq \tau(e^{\abs{x}}(\mu(\infty,x),\infty))=\tau(r)$. Hence 
\[
\mu(x)=\mu(x)\chi_{[0,\tau(r))}+\mu(\infty,x)\chi_{[\tau(r),\infty)}=\mu(x_0)+\mu(\infty,x)\chi_{[\tau(r),\infty)},
\]
and 
\[
V\mu(x)=V\mu(x_0)+\mu(\infty,x)V\chi_{[\tau(r),\infty)}=x_0+\mu(\infty,x)V\chi_{[\tau(r),\infty)}=\abs{x}r+\mu(\infty,x)V\chi_{[\tau(r),\infty)}.
\]

\end{proof}

\section{Main Results}
Throughout the remainder of the paper $E$ will stand for quasi-normed  space on $I$ whose dual $E^*$ separates points in $E$.

In view of Lemma \ref{lem:quasi}, the pairs $(\nonsp, \|\cdot\|_{\nonsp})$ and $(C_E, \|\cdot\|_{C_E})$ are quasi-normed spaces. The goal of this article is to find their Banach envelops.

Recall that given two $\sigma$-finite measure spaces $(\Omega_1 ,\mu_1 )$ and $(\Omega_2 , \mu_2 )$ we call a map $\sigma$ from $\Omega_1$ into $\Omega_2$ a measure preserving transformation  whenever for $\mu_2$-measurable subset $E$ of $\Omega_2$, the set
$\sigma^{-1} [E] = \{u \in\Omega_1:\, \sigma (u) \in E\}$ is a $\mu_1$-measurable subset of $\Omega_1$ and $\mu_1(\sigma^{-1} [E]) = \mu_2 (E)$ \cite{BS, KPS}.

It was shown in \cite{kammast} that given a symmetric quasi-normed function space $(E,\|\cdot\|_E)$ it follows that $(E,\|\cdot\|_{\widehat{E}})$ is a symmetric normed space. We will prove analogous result for symmetric quasi-normed sequence spaces.

\begin{lemma}
\label{lm:seqsymm}
If $(E,\|\cdot\|_E)$ is a  quasi-normed symmetric sequence space and $E\subseteq c_0$ then $(E,\|\cdot\|_{\widehat{E}})$ is a  normed symmetric sequence space.
\end{lemma}
\begin{proof}
It requires standard argument to show that if $f\in L_0(\mathbb N)$, $g\in (E,\|\cdot\|_{\widehat{E}})$, and $|f|\leq |g|$ then $f\in (E,\|\cdot\|_{\widehat{E}})$ and $\|f\|_{\widehat{E}}\leq \|g\|_{\widehat{E}}$. Therefore it suffices to prove that for two non-negative sequences $f=(f(k)),\ g=(g(k))\in E$ with $\mu(f)=\mu(g)$ we have $\|f\|_{\widehat{E}}=\|g\|_{\widehat{E}}$. Let $\epsilon>0$ and a sequence $\{g_i\}_{i=1}^n\subseteq E$ be such that $g=\sum_{i=1}^n g_i$ and $\sum_{i=1}^n \norme{g_i}\leq \|g\|_{\widehat{E}}+\epsilon$. Since $0\leq g\leq \sum_{i=1}^n|g_i|$,  we can always find sequences $h_i$ such that $g=\sum_{i=1}^nh_i$, $\supp{h_i}\subseteq \supp{g}$, and $0\leq h_i\leq |g_i|$. Since also $\sum_{i=1}^n \norme{h_i}\leq \sum_{i=1}^n \norme{g_i}\leq \|g\|_{\widehat{E}}+\epsilon$ we can assume without loss of generality that $g_i\geq 0$ and $\supp{g_i}\subseteq \supp{g}$. Since $f$ and $g$ have the same decreasing rearrangements and $f,g\in c_0$, their supports  must have the same measure and the following two cases hold.

Case $1^0$. Let  $m(\supp{g})=m(\supp{f})<\infty$. Then there exists  a bijective transformation $\sigma:\mathbb N\to\mathbb N$ such that $f=g\circ \sigma$. Hence $f=\sum_{i=1}^n g_i\circ \sigma$ and  $\|f\|_{\widehat{E}}\leq \sum_{i=1}^n\|g_i\circ \sigma\|_E=\sum_{i=1}^n\|g_i\|_E\leq \|g\|_{\widehat{E}}+\epsilon$. Consequently $\|f\|_{\widehat{E}}\leq \|g\|_{\widehat{E}}$. 

Case $2^0$. Let $m(\supp{g})=m(\supp{f})=\infty$. Then there is a bijective transformation $\sigma:\supp{f}\to\supp{g}$ such that $g(\sigma(k))=f(k)$ for $k\in\supp{f}$. Set $\tilde{g_i}(k)=g_i(\sigma(k))$ if $k\in\supp{f}$ and $\tilde{g_i}(k)=0$ otherwise. Then $f=\sum_{i=1}^n \tilde{g_i}$. Moreover for $t\geq 0$,  
\begin{align*}
d(t,\tilde{g_i})&=m\{k\in\mathbb N:\, \tilde{g_i}(k)>t\}=m\{k\in\supp{f}:\, g_i(\sigma(k))>t\}=m\{\sigma^{-1}(j):\,g_i(j)>t\}\\
&=m\left(\sigma^{-1}\{j\in\supp{g}:\,g_i(j)>t\}\right)=m\{j\in\supp{g}:\,g_i(j)>t\}\\
&=m\{j\in\mathbb N:\,g_i(j)>t\}=d(t,g_i).
\end{align*}
 The last equality for measures follows by the fact that $\supp{g_i}\subseteq \supp{g}$. Hence $\tilde{g_i}$ and $g_i$ have the same decreasing rearrangements, and so $\norme{\tilde{g_i}}=\norme{g_i}$. Consequently $\|f\|_{\widehat{E}}\leq \sum_{i=1}^n\norme{\tilde{g_i}}=\sum_{i=1}^n\norme{g_i}\leq \|g\|_{\widehat{E}}+\epsilon$. Therefore $\|f\|_{\widehat{E}}\leq \|g\|_{\widehat{E}}$. 

Similarly one can show that $\|g\|_{\widehat{E}}\leq \|f\|_{\widehat{E}}$, proving that $(E,\|\cdot\|_{\widehat{E}})$ is a normed symmetric sequence space. 
\end{proof}

 A sequence $\{x_n\}\subseteq S\Mtau$ converges in measure to $x\in S\Mtau$, denoted by $x_n\xrightarrow{\tau} x$, if and only if  $\mu(t, x-x_n)\to 0$ for all $t>0$  \cite[Lemma 3.1]{FK}.
For a symmetric normed space $E$,  we have that $\nonsp$ is continuously embedded in $S\Mtau$ with respect to the norm topology on $\nonsp$, that is if $\|x_n\|_{\nonsp}\to 0$,  it follows that $x_n\xrightarrow{\tau} 0$ \cite[Proposition 2.2]{DDP4}. 

In particular we get the analogous result for a normed symmetric space $E$ on $I$.
\begin{lemma}
\label{lm:measuretop}
Let $E$ be a   normed symmetric space on $I$. Then for the sequence $\{f_n\}\subseteq E$ satisfying $\|f_n\|_E\to 0$ it follows that $f_n\xrightarrow{m} 0$, that is $f_n$  converges to $0$ in measure $m$ on $I$.
\end{lemma}


Recall that for any Banach space $Z$ any
bounded linear operator $T:X \to Z$  extends with preservation of norm to an operator $\widehat{T}:\widehat{X}\to Z$ \cite{K3}. Given a  quasi-normed symmetric function (resp. sequence) space $(E,\|\cdot\|_E)$ it follows that $(E,\|\cdot\|_{\widehat{E}})$ is also a  normed symmetric function (resp. sequence)  space \cite{kammast}  (resp. Lemma \ref{lm:seqsymm}). Consequently $\|x\|_{L_1+L_\infty}\leq C \|x\|_{\widehat{E}}$ (resp. $\|x\|_{\ell_\infty}\leq C\|x\|_{\widehat{E}}$) for all $x\in E$ (see \cite{KPS}),
and the inclusion map $id:E\to L_1+L_\infty$ (resp. $id:E\to \ell_\infty$) has the norm $\|id\|=C$. Therefore there exists an extension of $id$ to $\widehat{id}:\widehat{E}\to L_1+L_\infty$ (resp. $\widehat{id}:\widehat{E}\to \ell_\infty$) with the same norm, so $\|\widehat{id}( x)\|_{L_1+L_\infty}\leq C\|x\|_{\widehat{E}}$ (resp. $\|\widehat{id}( x)\|_{\ell_\infty}\leq C\|x\|_{\widehat{E}}$ ), $x\in \widehat{E}$. 

It is known however that $\widehat{id}$ does not have to be an injection. Pietsch in \cite{Pt} shows an easy example of sequence spaces that this phenomenon occurs.
Unfortunately the same applies in the case of symmetric quasi-normed spaces as shown in \cite{DN}. The authors constructed there the Orlicz  sequence space, whose Banach envelope cannot be naturally treated as a sequence space. 

We state below the condition $(HC)$ characterizing quasi-normed symmetric spaces for which the embedding $\widehat{id}$ is injective, and therefore $\widehat{E}$ can be treated as a subspace of $L_1+L_{\infty}$ or $\ell_{\infty}$. The similar condition appeared  in \cite{DN} in the context of Orlicz sequence spaces.
   We will show in  Proposition \ref{prop:symmfun} that the property $(HC)$ guarantees that $\widehat{E}$ is also symmetric. 

\begin{definition}
\label{def:star}
Let $(E,\|\cdot\|_E)$ be a quasi-normed symmetric space.  We will write that $E\in (HC)$ if for  the Cauchy sequence $\{f_n\}$ in $(E,\|\cdot\|_{\widehat{E}})$ with $f_n\xrightarrow{m} 0$, it follows  that  $\|f_n\|_{\widehat{E}}\to 0$.
\end{definition} 
\begin{proposition}
\label{prop:star}
Let $(E,\|\cdot\|_E)$ be a quasi-normed symmetric space.
The inclusion map $id:E\to L_1+L_{\infty}$ extends to a continuous linear injection from $\widehat{E}$ into $L_1+L_{\infty}$ if and only if $E\in (HC)$.
\end{proposition}
\begin{proof}
 Suppose that $E\in (HC)$. As mentioned above, it remains only to show that $\widehat{id}$ is injective. Let $\widehat{id}(f)=0$, where $f\in \widehat{E}$. Choose a sequence $\{f_n\}\subseteq E$ for which $\|f_n-f\|_{\widehat{E}}\to 0$. Then 
 \[
\|f_n\|_{L_1+L_\infty}=\|\widehat{id}(f_n)-\widehat{id}(f)\|_{L_1+L_\infty}\leq C \|f_n-f\|_{\widehat{E}}\to 0.
 \]
 Therefore by Lemma \ref{lm:measuretop}, $f_n\xrightarrow{m} 0$. Since $\{f_n\}$ is also Cauchy in $\|\cdot\|_{\widehat{E}}$ norm, the fact that $E\in (HC)$ implies that $\|f_n\|_{\widehat{E}}\to 0$. Thus $f=0$ and  $\widehat{id}$ is injective.
 
 For the converse, suppose that $\widehat{id}$ is injective and let $\{f_n\}$ be a Cauchy sequence in $\|\cdot\|_{\widehat{E}}$ with $f_n\xrightarrow{m} 0$. Let $f\in \widehat{E}$ be such that $\|f_n-f\|_{\widehat{E}}\to 0$. Then 
 \[
 \|f_n-\widehat{id}(f)\|_{L_1+L_\infty}=\|\widehat{id}(f_n)-\widehat{id}(f)\|_{L_1+L_\infty}\leq C \|f_n-f\|_{\widehat{E}}\to 0,
 \]
  and $f_n\xrightarrow{m} \widehat{id}(f)$. Hence $\widehat{id}(f)=0$ and since $\widehat{id}$ is injective $f=0$. Consequently $\|f_n\|_{\widehat{E}}\to 0$ and $E\in (HC)$.
\end{proof}
As we see in the next proposition the large class of quasi-normed symmetric spaces satisfies condition $(HC)$.
\begin{proposition}
\label{prop:starorder}
If $E$ is an order continuous quasi-normed symmetric space then $E\in (HC)$.
\end{proposition}
\begin{proof}
Observe first that since $\|f\|_{\widehat{E}}\leq \|f\|_E$ for all $f\in E$, if $(E,\|\cdot\|_E)$ is order continuous then so is $(E,\|\cdot\|_{\widehat{E}})$. Suppose next that $\{f_n\}\subseteq E$ is Cauchy in $\|\cdot\|_{\widehat{E}}$, $f_n\xrightarrow{m} 0$ but $\|f_n\|_{\widehat{E}}$ does not converge to $0$. Thus without loss of generality we can assume that for all $n\in\mathbb N$, $C_1\leq\|f_n\|_{\widehat{E}}\leq C_2$ for some constants $C_1, C_2>0$. By Proposition 2.5 in \cite{kamray1} there is  a sequence $\{g_n\}$ of pairwise disjoint elements in $E$ such that 
\[
\sum_{n=1}^\infty\|f_n-g_n\|_{\widehat{E}}<\infty.
\] 
Note that the construction of the sequence $\{g_n\}$ in the proof of Proposition 2.5 in \cite{kamray1} does not require $(E,\|\cdot\|_{\widehat{E}})$ to be complete.

We have now that $\|f_n-g_n\|_{\widehat{E}}\to 0$. Thus going to a subsequence if necessary, we get a semi-normalized sequence  $\{g_n\}\subset E$, that is $C_1'\leq \|g_n\|_{\widehat{E}}\leq C_2'$, for some  $C_1', C_2'>0$ and all $n\in\mathbb N$, which is also Cauchy in $\|\cdot\|_{\widehat{E}}$. Thus, since $(E,\|\cdot\|_{\widehat{E}})$ is symmetric, we have
\[
C_1'\leq \|g_n\|_{\widehat{E}}\leq \||g_n|+|g_m|\|_{\widehat{E}}=\|g_n-g_m\|_{\widehat{E}}\to 0,
\]
which leads to a contradiction. Therefore $\|f_n\|_{\widehat{E}}\to 0$ and $E\in (HC)$.
\end{proof}

Note that not only order continuous spaces $E$ satisfy condition $(HC)$. In \cite{DN} in Example 5.7 (b) the  Orlicz space $\ell_\varphi$ is not order continues since  $\varphi$ does not satisfy condition $\Delta_2$ stated in the paper, but still $\varphi$ fulfills condition 5.5(c) which implies that the extension $\widehat{id}$ is an injection. Thus our main results Theorem \ref{thm:main} and Theorem \ref{thm:2} are not reduced only to order continuous spaces.

In view of Proposition \ref{prop:star}, if $E\in (HC)$ then we can identify $\widehat{id}(x)$ with $x$ for any $x\in \widehat{E}$. Thus any $x\in \widehat{E}$ will be treated as a measurable function (resp. sequence) which satisfies  $\|x\|_{L_1+L_\infty}\leq C\|x\|_{\widehat{E}}$ (resp. $\|x\|_{\ell_\infty}\leq C\|x\|_{\widehat{E}}$).
Observe that in view of Lemma \ref{lm:measuretop} if $\|x_n\|_{\widehat{E}}\to 0$, $x_n\in \widehat{E}$, then $x_n\xrightarrow{m} 0$. 

The next  simple observation is included here for easy reference.

\begin{lemma}
\label{lm:02}
Let $(E,\|\cdot\|_E)$ be a quasi-normed symmetric space  with $E\in (HC)$. If $\{f_n\}$ is Cauchy in $(E,\|\cdot\|_{\widehat{E}})$ and $f_n\xrightarrow{m} f$ for some $f\in L_0(I)$, then $f\in\widehat{E}$ and $\|f-f_n\|_{\widehat{E}}\to 0$.
\end{lemma}
\begin{proof}
Since  $\{f_n\}$ is Cauchy in $(E,\|\cdot\|_{\widehat{E}})$ there is $g\in\widehat{E}$ such that $\|f_n-g\|_{\widehat{E}}\to 0$. We have then by Lemma \ref{lm:measuretop} that $f_n\xrightarrow{m} g$ and thus $f=g\in \widehat{E}$.
\end{proof}

\begin{proposition}
\label{prop:symmfun}
 Let $E\in (HC)$ be a  quasi-normed symmetric space on $I$, where $E\subseteq c_0$  whenever $I=\mathbb N$. Then $\widehat{E}$ is a  Banach symmetric space. 
\end{proposition}
\begin{proof}
We will show first that $\widehat{E}$ is an ideal in $L_0(I)$ that is if $\abs{f}\leq \abs{g}$ a.e., where $g\in \widehat{E}$ and $f\in L_0(I)$ it follows that $g\in\widehat{E}$ and $\|f\|_{\widehat{E}}\leq \|g\|_{\widehat{E}}$. 

Choose a sequence $\{g_n\}\subseteq E$ such that $\|g_n-g\|_{\widehat{E}}\to 0$.  Consequently, $g_n\xrightarrow{m} g$ and $\abs{g_n}\xrightarrow{m} \abs{g}$.  Using the fact that $(E,\|\cdot\|_{\widehat{E}})$ is symmetric (\cite{kammast} for function spaces or Lemma \ref{lm:seqsymm} for sequence spaces) and $\abs{\abs{g_n}-\abs{g_m}}\leq \abs{g_n-g_m}$, $\{\abs{g_n}\}$ is Cauchy in $\|\cdot\|_{\widehat{E}}$. Hence by Lemma \ref{lm:02}, $\abs{g}\in\widehat{E}$ and  $\|\abs{g_n}-\abs{g}\|_{\widehat{E}}\to 0$. Thus $\|g\|_{\widehat{E}}=\lim_{n\to\infty}\|g_n\|_{\widehat{E}}=\lim_{n\to\infty}\|\abs{g_n}\|_{\widehat{E}}=\|\abs{g}\|_{\widehat{E}}$. So we showed that 
$\|g\|_{\widehat{E}}=\|\abs{g}\|_{\widehat{E}}$ for $g\in \widehat{E}$. 

 Now set $g_n'=\abs{g_n}\wedge \abs{g}$. Then in view of $0\leq g_n'\leq \abs{g_n}$ and $\abs{g_n}\in E$, we have that $g_n'\in E$. It is easy to check that  $0\leq \abs{g}-g_n'\leq \abs{\abs{g}-\abs{g_n}}$ and $\abs{g_n'-g_m'}\leq \abs{\abs{g_n}-\abs{g_m}}$. Hence $g_n'\xrightarrow{m} \abs{g}$ and $\{g_n'\}$ is Cauchy in $\|\cdot\|_{\widehat{E}}$. 
 Define next $f_n=\abs{f}\wedge g_n'$. Then $0\leq f_n\leq g_n'$ and $f_n\in E$. We have that for any $t\in I$,
\[
0\leq (\abs{f}-f_n)(t)=\left.\begin{cases}0 &\text{ if }\, \abs{f}(t)\leq g_n'(t)\\ \abs{f}(t)-g_n'(t) &\text{ if }\, \abs{f}(t)> g_n'(t)\end{cases}\right\}\,\leq \abs{g}(t)-g_n'(t).
\]
Hence $f_n\xrightarrow{m} \abs{f}$. Furthermore $\abs{f_n-f_m}\leq\abs{ g_n'-g_m'}$, which implies that $\{f_n\}$ is Cauchy in $\|\cdot\|_{\widehat{E}}$.
By Lemma \ref{lm:02},  $\abs{f}\in\widehat{E}$ and $\|f_n-\abs{f}\|_{\widehat{E}}\to 0$.  Consider next functions $f_n\,\text{sgn}(f)$, where $\text{sgn}(f)(t)=1$ whenever $f(t)\geq 0$ and $\text{sgn}(f)(t)=-1$ otherwise. Then $\abs{f_n\,\text{sgn}(f)}= \abs{f_n}$, where $f_n\in E$, and therefore $f_n\,\text{sgn}(f) \in E$. Moreover $f_n\,\text{sgn}(f)\xrightarrow{m} \abs{f}\,\text{sgn}(f)=f$. Since $\{f_n\,\text{sgn}(f)\}$ is also Cauchy in $\|\cdot\|_{\widehat{E}}$,  it follows that $f\in\widehat{E}$ and $\|f_n\,\text{sgn}(f)-f\|_{\widehat{E}}\to 0$.    

Finally, $\|f\|_{\widehat{E}}=\lim_{n\to\infty}\|f_n\,\text{sgn}(f)\|_{\widehat{E}}=\lim_{n\to\infty}\|f_n\|_{\widehat{E}}\leq \lim_{n\to\infty}\|\abs{g_n}\|_{\widehat{E}}=\|\abs{g}\|_{\widehat{E}}=\|g\|_{\widehat{E}}$.

It remains now to show that $(\widehat{E},\|\cdot\|_{\widehat{E}})$ is symmetric, that is $f\in \widehat{E}$ if and only if $\mu(f)\in\widehat{E}$, and $\|f\|_{\widehat{E}}=\|\mu(f)\|_{\widehat{E}}$. Suppose first that $f\in\widehat{E}$.   By the first part $\widehat{E}$ is an ideal in $L_0(I)$, and so  $f\in\widehat{E}$ if and only if $|f|\in\widehat{E}$ and $\|f\|_{\widehat{E}}=\|\abs{f}\|_{\widehat{E}}$. Hence we will assume that $f\geq 0$.
Let $\{f_n\}\subseteq E$ be such that $\|f-f_n\|_{\widehat{E}}\to 0$, and so 
$f_n\xrightarrow{m} f$.

There is a measure preserving, injective and onto transformation $\sigma: A\to I$, such that $\mu(f)\circ \sigma =f\chi_{A}$, where $A=\{t:\, f(t)> \mu(\infty, f)\}$ if $m\{t:\, f(t)> \mu(\infty, f)\}=\infty$, or $A=\{t:\, f(t)\geq \mu(\infty, f)\}$ if $m\{t:\, f(t)> \mu(\infty, f)\}<\infty$ (compare Therorem 7.5 and Corollary 7.6  in Chapter II \cite{BS}, see also \cite{noncomm}). 
 
 Set $\tilde{f_n}=|f_n|\chi_A$. Then 
 \[
 \|\tilde{f_n}-f\chi_A\|_{\widehat{E}}=\||\tilde{f_n}-f\chi_A|\|_{\widehat{E}}\leq \|\abs{f_n-f}\|_{\widehat{E}}\to 0.
 \]
 Hence $\tilde{f_n}\xrightarrow{m}f\chi_A$ and since $\sigma^{-1}$ is also measure preserving  $\tilde{f_n}\circ \sigma^{-1}\xrightarrow{m}f\chi_A\circ \sigma^{-1}=\mu(f)$.
Note next that by  symmetry of $(E,\|\cdot\|_{\widehat{E}})$, it follows that $\tilde{f_n}\circ \sigma^{-1}\in E$ and $\|\tilde{f_n}\circ \sigma^{-1}-\tilde{f_m}\circ \sigma^{-1}\|_{\widehat{E}}\to 0$. Hence by Lemma \ref{lm:02}, 
\begin{equation}\label{eq:11}
\mu(f)\in\widehat{E} \ \ \ \ \text{and} \ \ \ \ \|\tilde{f_n}\circ \sigma^{-1}-\mu(f)\|_{\widehat{E}}\to 0.
\end{equation}
 Similarly one can show that if $\mu(f)\in\widehat{E}$ then  $f\in\widehat{E}$.

Observe that if $\mu(\infty, f)=0$, in particular if $E\subseteq c_0$ is a sequence space, then  $A=\supp{f}$ and so $f\chi_A=f$. Using the symmetry of $(E,\|\cdot\|_{\widehat{E}})$ it follows  that  $\|f\|_{\widehat{E}}=\lim_{n\to\infty} \|\tilde{f_n}\|_{\widehat{E}}=\lim_{n\to\infty} \|\mu(\tilde{f_n})\|_{\widehat{E}}=\lim_{n\to\infty} \|\mu(\tilde{f_n}\circ \sigma^{-1})\|_{\widehat{E}}=\lim_{n\to\infty} \|\tilde{f_n}\circ \sigma^{-1}\|_{\widehat{E}}=\|\mu(f)\|_{\widehat{E}}$.

Consider now the case when $E$ is a function space and $\mu(\infty, f)>0$. Let $\epsilon>0$. Since $f$ and $f\chi_A$ are equimeasurable, applying \cite[Ch. II, Theorem 2.1]{KPS} there exist a
measure preserving transformation $\sigma_1$ on $[0,\tauone)$, and a function $g\in L_0(I)$ satisfying  $\abs{g}\leq 1$ a.e. such that 
\[
\|f-g\cdot(f\chi_A\circ \sigma_1)\|_{L_1\cap L_\infty}\leq {\epsilon}/{K},
\]
where the constant $K>0$ comes from the inequality $\|h\|_{\widehat{E}}\leq K\|h\|_{L_1\cap L_\infty}$, $h\in L_1\cap L_\infty$.
Hence $\|f-g\cdot (f\chi_A\circ \sigma_1)\|_{\widehat{E}}\leq \epsilon$.  By the similar argument applied above to show (\ref{eq:11})  one can prove that $\|\tilde{f_n}\circ \sigma_1-f\chi_A\circ \sigma_1\|_{\widehat{E}}\to 0$. This combined with the symmetry of $(E,\|\cdot\|_{\widehat{E}})$ and the fact that $(\widehat{E},\|\cdot\|_{\widehat{E}})$ is an ideal, implies that
\begin{align*}
\|f\|_{\widehat{E}}&\leq \|f-g(f\chi_A\circ \sigma_1)\|_{\widehat{E}}+\|g(f\chi_A\circ \sigma_1)\|_{\widehat{E}}\leq \epsilon
+\|f\chi_A\circ \sigma_1\|_{\widehat{E}}=\epsilon+\lim_{n\to\infty}\|\tilde{f_n}\circ \sigma_1\|_{\widehat{E}}\\
&=\epsilon+\lim_{n\to\infty}\|\tilde{f_n}\|_{\widehat{E}}=\epsilon+\|f\chi_A\|_{\widehat{E}},
\end{align*}
and so $\|f\|_{\widehat{E}}\leq \|f\chi_A\|_{\widehat{E}}$. Consequently, $\|f\|_{\widehat{E}} = \|f\chi_A\|_{\widehat{E}}$ and by (\ref{eq:11}) we get
\[
\|f\|_{\widehat{E}}=\lim_{n\to\infty}\|\tilde{f_n}\|_{\widehat{E}}=\lim_{n\to\infty}\|\tilde{f_n}\circ \sigma^{-1}\|_{\widehat{E}}=\|\mu(f)\|_{\widehat{E}}.
\]
\end{proof}

After proving the necessary results in function spaces we are ready to investigate the spaces of measurable operators.

The first result guarantees monotonicity of the norm $\|\cdot\|_{\widehat{\nonsp}}$ on $\nonsp$.
\begin{proposition}
\label{prop:1} Let $E$ be a quasi-normed symmetric function space. Then for any $0\leq x\leq y\in{\nonsp}$ we have $\|x\|_{\widehat{\nonsp}}\leq\|y\|_{\widehat{\nonsp}}$.
\end{proposition}
\begin{proof}Let  $0\le x \le y \in \nonsp$.
By \cite[Proposition 4.5]{Pag} there exists an operator $z\in \M$, $\|z\|_{\M}\leq 1$, such that $\sqrt x=\sqrt yz$. Hence $x=(\sqrt x)^*\sqrt x=z^*(\sqrt y)^*\sqrt y z=z^*yz$.
 Take  $\epsilon>0$ and $\{y_i\}_{i=1}^n\subseteq\nonsp$ such that $y=\sum_{i=1}^n y_i$ and
\[
\sum_{i=1}^n \normcomm{y_i}\leq \|y\|_{\widehat{\nonsp}}+\epsilon.
\] 
Then $x=z^* yz=\sum_{i=1}^n z^* y_i z$, where $\mu(z^* y_i z)\leq \|z\|_{\M}\|z^*\|_{\M} \mu( y_i)\leq  \mu( y_i)$.  Hence
\[
\|x\|_{\widehat{\nonsp}}\leq \sum_{i=1}^n \normcomm{z^*y_i z}\leq \sum_{i=1}^n \normcomm{y_i}\leq \|y\|_{\widehat{\nonsp}}+\epsilon.
\]
Since $\epsilon$ is arbitrary the claim follows.
\end{proof}

The first main result computing the norm $\|x\|_{\widehat{E\Mtau}}$ for $x\in\nonsp$ does not require any additional assumption on $E$. 
\begin{theorem}
\label{thm:1}
Let $\mathcal{M}$ be a non-atomic von Neumann algebra with a $\sigma$-finite trace $\tau$ and $E$ be a quasi-normed symmetric function space.  Then for every $x\in \nonsp$,
\[
\|x\|_{\widehat{\nonsp}}=\|\mu(x)\|_{\widehat{E}}.
\]
\end{theorem}
\begin{proof}
Observe first that since $(E,\|\cdot\|_{\widehat{E}})$ is a symmetric normed space \cite{kammast}, given $x\in\nonsp$ the functional $|||x|||=\|\mu(x)\|_{\widehat{E}}$ is a norm on $\nonsp$ by \cite[Theorem 8.7]{KS}. Moreover, for any $x\in\nonsp$ we have that
\begin{equation}
\label{eq:thm1.1}
|||x|||=\|\mu(x)\|_{\widehat{E}}\leq \norme{\mu(x)}=\normcomm{x}.
\end{equation}
Since $\|\cdot\|_{\widehat{\nonsp}}$ is the largest seminorm satisfying (\ref{eq:thm1.1}), it follows that
\[
|||x|||=\|\mu(x)\|_{\widehat{E}}\leq\|x\|_{\widehat{\nonsp}}\ \ \ \ \text{for all $x\in\nonsp$}.
\]
For the converse let $x\in \nonsp$. We first observe that $\|x\|_{\widehat{\nonsp}}=\|\abs{x}\|_{\widehat{\nonsp}}$. To see this let $x=\sum_{i=1}^n x_i$ for some $\{x_i\}_{i=1}^n\subseteq \nonsp$. By the polar decomposition of $x$ there exists a partial isometry $u$  such that $x=u\abs{x}$ or equivalently $u^*x=\abs{x}$.  We have that $\abs{x}=\sum_{i=1}^n u^*x_i$ and so 
 \[
 \|\abs{x}\|_{\widehat{\nonsp}}\leq \sum_{i=1}^n\normcomm{u^*x_i}\leq \sum_{i=1}^n\normcomm{x_i}.
 \]
Consequently, $ \|\abs{x}\|_{\widehat{\nonsp}}\leq  \|x\|_{\widehat{\nonsp}}$. Similarly one can show the opposite inequality.

Therefore we will assume without loss of generality that $x\geq 0$.
Let $\epsilon>0$ and choose $\{f_i\}_{i=1}^n\subseteq E$  such that 
\begin{equation}
\label{eq:thm1.2}
\mu(x)=\sum_{i=1}^n f_i\quad \text{and}\quad \sum_{i=1}^n \norme{f_i} \le \|\mu(x)\|_{\widehat{E}}+\epsilon.
\end{equation}

By Corollary \ref{cor:isom} there exists a unital $*$-isomorphism $V$ acting from the $*$-algebra $S\left(\left[0,\tauone\right),m\right)$ into $S(\mathcal{N},\tau)$ with $\mathcal{N}\subseteq q\M q$, such that 
\[ V\mu(x)=xr+\mu(\infty, x)V\chi_{[\tau(r),\infty)}\ \ \ \text{ and }\ \ \ \mu(Vf)=\mu(f)\ \ \text{ for all } f\in S\left([0,\tauone),m\right),
\]
where $r=e^{x}(\mu(\infty,x),\infty)$. The projection $q=\one$ if $\tau(r)<\infty$ and $q=r$ otherwise.

Note that  $\mu(t,x)=\mu(\sum_{i=1}^n t/n,\sum_{i=1}^n f_i)\leq \sum_{i=1}^n \mu(t/n,f_i)$, and so $\mu(\infty,x)\leq \sum_{i=1}^n\mu(\infty,f_i)$. Therefore we can select constants $C_i\ge 0$ satisfying 
$\sum_{i=1}^n C_i=\mu(\infty,x)$ and $C_i\leq \mu(\infty,f_i)$, $i=1,2,\dots, n$.

We have $0\leq xr\leq V\mu(x)$ and $0\leq xr=r(xr)r\leq rV\mu(x)r$. Hence in view of (\ref{eq:thm1.2}), 
\begin{align*}
0\leq x&= xr+ xr^\perp\leq rV\mu(x)r+\mu(\infty,x)r^\perp=\sum_{i=1}^n \left(rVf_i r+C_ir^\perp\right).
\end{align*}
Since $r=s(xr)$, and $xr\in S(\mathcal{N},\tau)$ by Corollary  \ref{cor:isom}, we have that $r\in \mathcal{N}$. Hence there is $A\subseteq [0,\tauone)$ such that $V\chi_A=r$. Observe that $0\leq rVf_i r+C_ir^\perp \leq V(f_i \chi_A)+\mu(\infty, f_i) V\chi_{A^c}=V(f_i \chi_A+\mu(\infty, f_i)\chi_{A^c})$, and so $\mu(rVf_i r+C_ir^\perp )\leq \mu(f_i \chi_A+\mu(\infty, f_i)\chi_{A^c})\leq \mu(f_i)$.
Therefore by Proposition \ref{prop:1},
\[
\|x\|_{\widehat{\nonsp}}\leq \|rV\mu(x)r+\mu(\infty,x)r^\perp\|_{\widehat{\nonsp}}\leq \sum_{i=1}^n\normcomm{rVf_i r+C_ir^\perp}\leq  \sum_{i=1}^n\norme{f_i}\leq \|\mu(x)\|_{\widehat{E}}+\epsilon.
\] 
Since $\epsilon$ is arbitrary the claim follows.

\end{proof}

\begin{proposition}
\label{prop:final1}
 Let $\mathcal{M}$ be a non-atomic von Neumann algebra with a $\sigma$-finite trace $\tau$ and $E\in (HC)$ be a quasi-normed symmetric function space.  Then
\[
\left(\widehat{\nonsp},\|\cdot\|_{\widehat{\nonsp}}\right)\subseteq\left(\widehat{E}\Mtau,\|\cdot\|_{\widehat{E}\Mtau}\right),
\]
and for all $x\in \widehat{E\Mtau}$
\[\quad\|x\|_{\widehat{E\Mtau}}=\|x\|_{\widehat{E}\Mtau}.\]
Consequently, the embedding of $\widehat{\nonsp}$ into $S\Mtau$ is continuous with respect to the norm topology on $\widehat{\nonsp}$.

\end{proposition}
\begin{proof}
By Proposition \ref{prop:symmfun}, $\widehat{E}$ is asymmetric Banach function space, and so $\widehat{E}\Mtau$ is a Banach space \cite{KS}.
Since $E\subseteq \widehat{E}$ we have that $\nonsp\subseteq \widehat{E}\Mtau$. Moreover, by Theorem \ref{thm:1} we have that  $\|x\|_{\widehat{E}\Mtau}= \|x\|_{\widehat{\nonsp}}$ for all $x\in\nonsp$.
Thus  completion of $\nonsp$ with respect to the norm $\|\cdot\|_{\widehat{\nonsp}}$  is the same as the closure of $\nonsp$ in $\widehat{E}\Mtau$. Consequently,
\[
\widehat{\nonsp}\subseteq\widehat{E}\Mtau\quad\text{ and }\quad  \|x\|_{\widehat{E}\Mtau}= \|x\|_{\widehat{\nonsp}}\quad\text{ for all }\quad x\in\widehat{\nonsp}.
\]
The second part follows from the well known fact that since  $\widehat{E}$ is a Banach symmetric space then  $\widehat{E}\Mtau$ is continuously embedded into $S\Mtau$ \cite[Proposition 2.2]{DDP4}.
\end{proof}
\begin{lemma}
\label{lm:03}
Suppose that $\tau$ is a $\sigma$-finite trace on a non-atomic von Neumann algebra $\M$,    and $E\in(HC)$ is a quasi-normed symmetric function space. If $\|x_n-x\|_{\widehat{E}\Mtau}\to 0$ for $x\in\widehat{E}\Mtau$ and $\{x_n\}\subseteq \widehat{\nonsp}$ then $x\in \widehat{\nonsp}$ and $\|x_n-x\|_{\widehat{\nonsp}}\to 0$.
\end{lemma}
\begin{proof}
Let $\|x_n-x\|_{\widehat{E}\Mtau}\to 0$, where $x\in\widehat{E}\Mtau$ and $\{x_n\}\subseteq \widehat{\nonsp}$. Then  $x_n\xrightarrow{\tau} x$  and by Proposition \ref{prop:final1}, $\|x_n-x_m\|_{\widehat{\nonsp}}=\|x_n-x_m\|_{\widehat{E}\Mtau}$, $m,n\in\mathbb N$. Hence $\{x_n\}$ is Cauchy in $\widehat{\nonsp}$ and there is $z\in\widehat{\nonsp}$ such that $\|z-x_n\|_{\widehat{\nonsp}}\to 0$. Again by Proposition \ref{prop:final1}, $x_n\xrightarrow{\tau} z$, and so $z=x$. Consequently, $x\in \widehat{\nonsp}$ and $\|x_n-x\|_{\widehat{\nonsp}}\to 0$.
\end{proof}
We observe next that $\widehat{\nonsp}$ is an ideal in $S\Mtau$. 
\begin{lemma}
\label{lm:idealoper}
Suppose that $\tau$ is a $\sigma$-finite trace on a non-atomic von Neumann algebra $\M$,    and $E\in(HC)$ is a quasi-normed symmetric function space. Let $x,y\in S\Mtau$ with $0\leq x\leq y$ and $y\in \widehat{\nonsp}$. Then $x\in\widehat{\nonsp}$ and $\|x\|_{\widehat{\nonsp}}\leq \|y\|_{\widehat{\nonsp}}$.
\end{lemma}
\begin{proof}
Notice first by Proposition \ref{prop:final1} we have that $\widehat{\nonsp}\subseteq S\Mtau$.
Let $x\in S\Mtau$, $y\in \widehat{\nonsp}$ and $0\leq x\leq y$. As noted in the proof of Proposition \ref{prop:1} there exists  $z\in \M$, $\|z\|_{\M}\leq 1$, such that $x=z^*yz$.
Let $\{y_n\}\subseteq \nonsp$ be such that $\|y-y_n\|_{\widehat{\nonsp}}\to 0$. By Proposition \ref{prop:final1},
\[
\|y-y_n\|_{\widehat{E}\Mtau}=\|y-y_n\|_{\widehat{\nonsp}}\to 0.
\]
 We have $z^*y_n z\in \nonsp\subset \widehat{E(\mathcal{M},\tau)}$ and $y\in \widehat{E}(\mathcal{M},\tau)$ by Proposition \ref{prop:final1}.  Thus $x=z^*yz\in \widehat{E}(\mathcal{M},\tau)$ and
\begin{align}
\label{eq:lmidealopereq1}
\|x-z^*y_nz\|_{\widehat{E}\Mtau}&=\|\mu\left(z^*(y-y_n)z\right)\|_{\widehat{E}}
\leq \|z\|_{\M}\|z^*\|_{\M}\|y-y_n\|_{\widehat{E}\Mtau}\to 0.
\end{align}
Hence by Lemma \ref{lm:03},  $x\in\widehat{\nonsp}$. The norm inequality follows instantly by Proposition \ref{prop:final1}, 
\[
\|x\|_{\widehat{\nonsp}}=\|x\|_{\widehat{E}\Mtau}\leq \|y\|_{\widehat{E}\Mtau}=\|y\|_{\widehat{\nonsp}}.
\]
\end{proof}

\begin{lemma}
\label{lm:complitionisom}
Let $E\in(HC)$ be a quasi-normed symmetric function space, and $\mathcal{M}$ be a non-atomic von Neumann algebra with a $\sigma$-finite trace $\tau$. Suppose that $V$ is a $*$-isomorphism from $S\left([0,\tauone),m\right)$ into  $S\Mtau$ such that $\mu(Vf)=\mu(f)$ for all $f\in S\left([0,\tauone),m\right)$. If $x\in\widehat{E}\Mtau$ then $V\mu(x)\in\widehat{\nonsp}$. 
\end{lemma}
\begin{proof}
Let $x\in \widehat{E}\Mtau$ that is $x\in S\Mtau$ and $\mu(x)\in \widehat{E}$. Then there is $\{f_n\}\subseteq E$ such that $\|\mu(x)-f_n\|_{\widehat{E}}\to 0$. Therefore
\[
\|V\mu(x)-Vf_n\|_{\widehat{E}\Mtau} = \|\mu(V\mu(x)-Vf_n)\|_{\widehat{E}}= \|\mu(\mu(x)-f_n)\|_{\widehat{E}}= \|\mu(x)-f_n\|_{\widehat{E}}\to 0.
\]
Moreover, $Vf_n\in S\Mtau$ and $\mu(Vf_n)=\mu(f_n)\in E$. Therefore  $\{Vf_n\}\subseteq \nonsp\subseteq \widehat{\nonsp}$, and by Lemma \ref{lm:03} it follows that $V\mu(x)\in \widehat{\nonsp}$.

\end{proof}

\begin{lemma}
\label{lm:absolutvalue}
Let $\tau$ be a $\sigma$-finite trace on a non-atomic von Neumann algebra $\M$ and $E\in(HC)$ be a quasi-normed symmetric function space. Then $x\in \widehat{\nonsp}$ if and only if $\abs{x}\in\widehat{\nonsp}$, and $\|x\|_{\widehat{\nonsp}} = \|\,|x|\,\|_{\widehat{\nonsp}}$.
\end{lemma}
\begin{proof}
Let $x\in\widehat{\nonsp}$ and $\{x_n\}\subseteq \nonsp$ be such that $\|x-x_n\|_{\widehat{\nonsp}}\to 0$. Then by Proposition \ref{prop:final1},  $x, x_n \in \widehat{E}\Mtau$, $n\in\mathbb{N}$. 
Letting $x=u\abs{x}$ be the polar decomposition of $x$,  $u^*x_n\in\nonsp$ and  $|x| = u^*x\in \widehat{E}\Mtau$. By Proposition \ref{prop:final1},
\[
\||x|-u^*x_n\|_{\widehat{E}\Mtau}=\|u^*x-u^*x_n\|_{\widehat{E}\Mtau}\leq \|x-x_n\|_{\widehat{E}\Mtau}=\|x-x_n\|_{\widehat{\nonsp}}\to 0.
\]
Since $\{u^*x_n\}\subseteq \nonsp\subseteq\widehat{\nonsp}$, Lemma \ref{lm:03} implies that $\abs{x}\in\widehat{\nonsp}$.
The converse can be proved similarly.

By Proposition \ref{prop:final1} we have the equality of norms, since $\|x\|_{\widehat{\nonsp}}=\|x\|_{\widehat{E}\Mtau}=\|\abs{x}\|_{\widehat{E}\Mtau}=\|\abs{x}\|_{\widehat{\nonsp}}$.
\end{proof}

\begin{proposition}
\label{prop:opersymm}
Let $\tau$ be a $\sigma$-finite trace on a non-atomic von Neumann algebra $\M$  and  $E\in(HC)$ be a quasi-normed symmetric function space. Then for $x,y\in S\Mtau$, with $y\in \widehat{\nonsp}$, and $\mu(x)\leq \mu(y)$ we have that $x\in \widehat{\nonsp}$ and $\|x\|_{\widehat{\nonsp}}\leq \|y\|_{\widehat{\nonsp}}$.
\end{proposition}
\begin{proof}
By Proposition \ref{prop:final1}, $\widehat{\nonsp}\subseteq \widehat{E}\Mtau\subseteq S\Mtau$. Let $\mu(x)\leq \mu(y)$, where $x,y\in S\Mtau$ and  $y\in\widehat{\nonsp}$. Since $\mu(x)=\mu(\abs{x})$ and $\mu(y)=\mu(\abs{y})$, Lemma \ref{lm:absolutvalue} allows to assume that $x,y\geq 0$.

Let $r=e^{\abs{x}}(\mu(\infty, x),\infty)$ and $q$ be like in Corollary \ref{cor:isom}. That is if $\tau(r)<\infty$ then $q=\one$, and if $\tau(r)=\infty$ then $q=r$. By Corollary \ref{cor:isom}, there is a $*$-isomorphism $V$ from $S\left([0,\tauone),m\right)$ into  $S(\mathcal{N},\tau)\subseteq qS\Mtau q$, preserving a decreasing rearrangement, and such that 
\[
V\mu(x)=xr+\mu(\infty,x)V\chi_{[\tau(r),\infty)}.
\]
Since $*$-isomorphism preserves the order,  $0\leq xr\leq V\mu(x)\leq V\mu(y)$. By Lemma \ref{lm:complitionisom} and by $y\in \widehat{E}\Mtau$ we have that $V\mu(y)\in \widehat{\nonsp}$. Thus Lemma \ref{lm:idealoper} implies that $xr\in \widehat{\nonsp}$.

Case 1. If $\mu(\infty, x)=0$ then $r=e^{\abs{x}}(0,\infty)=s(x)$ and $x=xr\in\widehat{\nonsp}$.

Case 2. If $\mu(\infty,x)> 0$, then $\tauone =\infty$ and since $\mu(x)\in \widehat{E}$ we must have $\chi_{[0,\infty)}\in \widehat{E}$. Hence $\one\in\widehat{E}\Mtau$. Let $V_1$ be a $*$-isomorphism from $S([0,\infty),m)$ to $S\Mtau$ such that $V_1\chi_{[0,\infty)}=\one$ and $\mu(V_1f)=\mu(f)$ for all $f\in S([0,\infty),m)$. Such an isomorphism can be chosen by applying Proposition \ref{prop:isom} to an operator $z\in S\Mtau$ with $\tau(s(z))<\infty$. Then by Lemma \ref{lm:complitionisom}, $\one=V_1\chi_{[0,\infty)}\in\widehat{\nonsp}$. Therefore in view of $0\leq r^\perp \leq \one$ and by Lemma \ref{lm:idealoper}, $r^\perp\in \widehat{\nonsp}$. Since $0\leq xr^\perp\leq \mu(\infty,x)r ^\perp$,
it follows that $xr^\perp \in \widehat{\nonsp}$. Consequently, $x=xr + xr^\perp\in\widehat{\nonsp}$.

Hence $x,y\in\widehat{\nonsp}$ and by Proposition \ref{prop:final1} and the fact that $\widehat{E}$ is symmetric we have
\[
\|x\|_{\widehat{\nonsp}}=\|\mu(x)\|_{\widehat{E}}\leq \|\mu(y)\|_{\widehat{E}}=\|y\|_{\widehat{\nonsp}}.
\]
\end{proof}

We are ready now for the the counterpart of Proposition \ref{prop:final1}.

\begin{proposition}
\label{prop:final2}
Let $\M$ be a non-atomic von Neumann algebra with a $\sigma$-finite trace $\tau$   and $E\in(HC)$ be a quasi-normed symmetric function space. Then
\[
\left(\widehat{E}\Mtau,\|\cdot\|_{\widehat{E}\Mtau}\right)\subseteq \left(\widehat{\nonsp},\|\cdot\|_{\widehat{\nonsp}}\right),
\]
and for all $x\in \widehat{E}\Mtau$,
\[\quad\|x\|_{\widehat{E\Mtau}}=\|x\|_{\widehat{E}\Mtau}.\]
\end{proposition}
\begin{proof}
Let $x\in \widehat{E}\Mtau$. By Lemma \ref{lm:complitionisom}, $V\mu(x)\in \widehat{\nonsp}$, where $V$ is a $*$ - isomorphism from $S\left([0,\tauone),m\right)$ into $S\Mtau$ such that $\mu(Vf)=\mu(f)$ for all $f\in S\left([0,\tauone),m\right)$. Since operators $x$ and $V\mu(x)$ have the same singular value function, in view of Proposition \ref{prop:opersymm}, $x\in \widehat{\nonsp}$.
The equality of norms follows now instantly by Proposition \ref{prop:final1}.
\end{proof}

As a consequence of Propositions \ref{prop:final1} and \ref{prop:final2} we  state our main result.

\begin{theorem}
\label{thm:main}
Let $\M$ be a non-atomic von Neumann algebra with a $\sigma$-finite trace $\tau$   and $E\in(HC)$ be a quasi-normed symmetric function space.  Then
\[
\left(\widehat{E}\Mtau,\|\cdot\|_{\widehat{E}\Mtau}\right)= \left(\widehat{\nonsp},\|\cdot\|_{\widehat{\nonsp}}\right)
\]
with equality of norms.
\end{theorem}
We give now the  variant of Theorem \ref{thm:1} for the unitary matrix spaces $C_E$. Recall that so far for the unitary matrix spaces the result has been obtained in \cite{naw} under additional  assumption that $E$ is weakly-separable. 
\begin{theorem}
\label{thm:2p1}
Let $E$ be a quasi-normed  symmetric sequence space  such that $E\subseteq c_0$. Then for all $x\in C_E$,
\[
 \|x\|_{\widehat{C_E}}= \|S(x)\|_{\widehat{E}}.
\]
 
\end{theorem}
\begin{proof}
Observe first that in view of  Proposition \ref{lm:seqsymm} and \cite[Corollary 8.8]{KS}, the functional $|||x|||=\|S(x)\|_{\widehat{E}}$, $x\in C_E$, is a norm on $C_E$. Moreover, $|||x|||=\|S(x)\|_{\widehat{E}}\leq \|S(x)\|_E=\|x\|_{C_E}$, $x\in C_E$. Since $\|\cdot\|_{\widehat{C_E}}$ is the strongest semi-norm on $C_E$ for which the latter is true, it follows that for all $x\in C_E$ we have $\|S(x)\|_{\widehat{E}}\leq \|x\|_{\widehat{C_E}}$. 

In order to show the converse inequality let $\epsilon >0$ and $\{a_i\}_{i=1}^n\subseteq E$ be such that $S(x)=\sum_{i=1}^n a_i$ and $\sum_{i=1}^n\|a_i\|_{E}\leq \|S(x)\|_{\widehat{E}}+\epsilon$. By \cite[Proposition 1.1]{A}  there is a $*$-isomorphism $V:E\to C_E$ such that $V S(x)=x$ and $S(V a)=\mu(a)$ for all $a\in E$. Therefore $V$ is an isometry on $E$. Hence $x=VS(x)=\sum_{i=1}^n V a_i$ and $\|x\|_{\widehat{C_E}}\leq \sum_{i=1}^n\|V a_i\|_{C_E}=\sum_{i=1}^n\|a_i\|_{E}\leq \|S(x)\|_{\widehat{E}}+\epsilon$. Consequently, $\|x\|_{\widehat{C_E}}\leq \|S(x)\|_{\widehat{E}}$ and combining with the previous inequality we get
\begin{equation}\label{eq:112}
\|x\|_{\widehat{C_E}}=\|S(x)\|_{\widehat{E}}\ \ \ \ \text{for all $x\in C_E$}.
\end{equation}
\end{proof}
\begin{theorem}
\label{thm:2}
Let $E\in(HC)$ be a quasi-normed  symmetric sequence space  such that $E\subseteq c_0$. Then 
\[
(\widehat{C_E}, \|\cdot\|_{\widehat{C_E}})=(C_{\widehat{E}}, \|\cdot\|_{C_{\widehat{E}}})
\]
 with equality of norms.
\end{theorem}
\begin{proof}
By Theorem \ref{thm:2p1}, we have that $(C_E,\|\cdot\|_{\widehat{C_E}})=(C_E,\|\cdot\|_{C_{\widehat{E}}})\subseteq (C_{\widehat{E}},\|\cdot\|_{C_{\widehat{E}}})$, and consequently 
\begin{equation}\label{eq:113}
(\widehat{C_E},\|\cdot\|_{\widehat{C_E}})=\overline{(C_E,\|\cdot\|_{C_{\widehat{E}}})}^{\|\cdot\|_{C_{\widehat{E}}}}\subseteq (C_{\widehat{E}},\|\cdot\|_{C_{\widehat{E}}})
\end{equation}
 with $\|x\|_{\widehat{C_E}}=\|x\|_{C_{\widehat{E}}}$ for all $x\in\widehat{C_E}$. 

To show the opposite inclusion, let $x\in C_{\widehat{E}}$, that is $x\in K(H)$ and $S(x)\in \widehat{E}$. Then there exists a sequence $\{a_n\}\subseteq E$ such that $\|a_n-S(x)\|_{\widehat{E}}\to 0$. Again by Arazy's result, let $V:\widehat{E}\to C_{\widehat{E}}$ be an isometry such that $VS(x)=x$ and $S(V a)=\mu(a)$ for all $a\in \widehat{E}$. In particular  we have that $V$ acts from $E$ into $C_E$ and so $Va_n \in C_E$, $n\in\mathbb{N}$. Then 
\[
\|V a_n-x\|_{C_{\widehat{E}}}=\|V a_n-VS(x)\|_{C_{\widehat{E}}}=\|a_n-S(x)\|_{\widehat{E}}\to 0, \ \ \ \ n\to\infty.
\] 
Hence and in view of (\ref{eq:112}) we have that 
\[
\|V a_n-V a_m\|_{\widehat{C_E}}=\|V a_n-Va_m\|_{C_{\widehat{E}}}\to 0 \ \ \ \text{as} \ \ \ n,m\to\infty,
\]
 and so $\{Va_n\}$ is Cauchy in $\|\cdot\|_{\widehat{C_E}}$.
Therefore there is $z\in\widehat{C_E}$ such that $\|Va_n-z\|_{\widehat{C_E}}\to 0$. Now by (\ref{eq:113}), $\|V a_n-z\|_{C_{\widehat{E}}}=\|V a_n-z\|_{\widehat{C_E}}\to 0$. Hence $x=z\in\widehat{C_E}$ and 
\[
\|x\|_{\widehat{C_E}}=\lim_{n\to\infty}\|V a_n\|_{\widehat{C_E}}=\lim_{n\to\infty}\|V a_n\|_{C_{\widehat{E}}}=\|x\|_{C_{\widehat{E}}}.
\]

\end{proof}

We conclude the paper with the corollaries on Banach envelopes of noncommutative $L_p\Mtau$ spaces and Schatten classes $C_p=\{x\in K(H):\, S(x)\in \ell_p\}$, $0<p<1$. They follow from Theorems \ref{thm:main} and \ref{thm:2}, as the spaces $L_p$ and $\ell_p$, $0<p<1$, are order continuous and therefore satisfying $(HC)$.
\begin{corollary}
 Let $\M$ be a non-atomic, semifinite von Neumann algebra with a faithful, normal, $\sigma$-finite trace $\tau$. For $0<p<1$ we have that
  \[
L_p\Mtau\widehat{\phantom{x}}=\{0\}.
 \]

\end{corollary}
\begin{corollary}\cite{naw}
For $0<p<1$ it follows that
\[
\widehat{C_p}=C_1\text{ with } \|x\|_{\widehat{C_p}}=\|S(x)\|_1, \, x\in \widehat{C_p}.
\]
\end{corollary}

\begin{center}
{\bf Acknowledgement}\\
\end{center} 
\textit{We are grateful to the reviewer for his/hers insightful  and detailed report of the paper, which allowed us to improve significantly the final version of the manuscript. }

\bibliographystyle{amsplain}

\end{document}